\documentclass{article}
\usepackage[utf8]{inputenc}
\usepackage{multirow}
\usepackage{amsmath,amssymb}
\usepackage{amsthm}
\usepackage{graphicx}
\usepackage{caption}
\usepackage{subcaption}
\usepackage{comment}
\usepackage{xcolor}
\usepackage{booktabs}
\usepackage{wrapfig}

\def\R{{\bf R}}

\def\eps{\varepsilon}

\def\argmin{\mathop{\rm arg\,min}}

\def\A{{\cal A}}
\def\PE{\mathop{\rm PE}}

\def\ricmat{\theta}

\def\nullbox#1#2#3{\setbox0=\null
          \ht0=#1 \dp0=#2 \wd0=#3 \box0 }
\def\ns#1{\nullbox{0}{0}{0}}

\theoremstyle{definition}
\newtheorem{example}{Example}

\theoremstyle{plain}
\newtheorem{theorem}{Theorem}

\def\slug{\hbox{\kern1.5pt\vrule width2.5pt height6pt depth1.5pt\kern1.5pt}}

\def\ac#1{\textcolor{red}{#1}}

\title{\bf
Forward Approximate Solution \\
for Linear Quadratic Tracking}
\author{Alessandro Betti$^\star$, Michele Casoni$^\circ$, and Marco Gori$^\circ$\\
$\star$ MAASAI - Univ. C\^{o}te d'Azur, $\circ$ SAILab - Univ. Siena}

\begin{document}

\maketitle

\begin{abstract}
In this paper, we discuss an approximation strategy for solving 
the Linear Quadratic Tracking that is both
forward and local in time. We exploit the known form of the
value function along with a time reversal transformation
that nicely addresses the boundary condition consistency. 
We provide the results of an experimental investigation 
with the aim of showing how the proposed solution 
performs with respect to the optimal solution. Finally, we 
also show that the proposed solution turns out to be a valid
alternative to model predictive
control strategies, whose computational burden is dramatically 
reduced. 
\end{abstract}

\section{Introduction}
The linear quadratic problem and the linear quadratic problem with tracking and
as well as a general deterministic optimal control problem can be regarded
as a variational problem for an integral functional~\cite{evans2010partial}, where 
the integration is usually performed over the temporal variable. 
As such the solutions to
such problems are characterized by non-locality in time that usually
manifests itself through the need to respect the boundary conditions. 
For instance, in the
LQ problem over a finite interval the Riccati equation must be solved
starting from a final condition that is directly inherited by the final
condition we have on the value function in the Hamilton-Jacobi-Bellman equation.
Likewise, when using the Hamilton-Jacobi equation this results into
a final condition for the \emph{costate}.

In many real-world applications, however, when the number of control variables is big and
the temporal horizon on which the problem is posed is large (or infinite), a
numerical determination of the optimal solution is not feasible, which gives
rise to the interest of finding approximations of the solution that dramatically reduce
the computational cost. 
A control strategy that is  typically carried out relies on the idea of breaking down the
optimization to limited temporal windows on which the exact solution can be
explicitly computed and then by gluing together the solutions on the different
temporal windows. This strategy is referred to as 
Model Predictive Control (MPC) (see~\cite{REBLE20121812}), where the optimal
solution is approximated through an an open-loop finite horizon.

In this paper we propose a different approximation strategy in
the case of Linear Quadratic with tracking (LQT problem) that is
inspired and motivated by the observation that in the LQ problem 
Riccati's differential equation is an autonomous system that, once we
perform a time-reversal transformation can be solved forward in time. 
This corresponds with carrying out a {\em SIgn Flip Transformation}
(SIFT) that gives rise to efficient and accurate solutions especially
when the search for the optimal control is mostly focussed on
precision of the solution.
We show that, at least in the scalar case, the solution to this new forward
problem asymptotically converges to the solution of the algebraic Riccati
equation that characterizes the optimal solution on the infinite-time
horizon. Interestingly, a similar approach can be employed for the LQT problem.

The resulting method leads to a local-in-time forward computational scheme that, 
unlike MPC, does not require a receding horizon approach.  
Moreover, the proposed method yields very good
results when compared to the optimal solution and to MPC approaches,
which have an significantly higher computational (see Section~\ref{Sec:exp}).

The  paper is organized as follows. 
In Section~\ref{Sec:HJB} we review the classic  Hamilton-Jacobi-Bellman equations
with the purpose of presenting the general case in which the system and 
the Lagrangian are explicitly time-dependent. 
In Section~\ref{Sec:LQ}
we showcase our approximation strategy for the LQ problem, while the extention
to the LQT problem is discussed in Section~\ref{sec:lqt}.
The experimental results are presented in Section~\ref{Sec:exp} and, 
finally, some conclusions are drawn in Section~\ref{Sec:conc}.

\section{Hamilton-Jacobi-Bellman equations}
\label{Sec:HJB}
Throughout the paper we consider describe the dynamics of the state variable
$s\mapsto y(s)\in\R^n$ in terms of the following Cauchy problem:
\begin{equation}
\begin{cases}
y'(s) = f(y(s),\alpha(s),s) & s\in(t,T) \\
y(t)=x,
\end{cases}
\label{ODE-sys}
\end{equation}
where $\alpha\in{\cal A}:=\{\alpha\colon [0,T]\to A\subset \R^m:
\alpha\hbox{ is measurable
}\}$, $x\in\R^n$ and $f$ is a Lipshitz function of its arguments, which
is sufficient to claim the existence of the solution of ODE~(\ref{ODE-sys}). 
The ``goodness'' of the control $\alpha$, once $x$ and $t$ are fixed
is measured through the cost functional
\begin{equation}
\alpha\in\A\mapsto C_{x,t}(\alpha):=\int_t^T \ell(\alpha(s),
y_{x,t}(s,\alpha),s)\,
ds +J(y_{x,t}(T,\alpha)),
\end{equation}
where $y_{x,t}(s,\alpha)$ solves~(\ref{ODE-sys}). We define the value function as
$v(x,t):= \inf_{\alpha\in\A} C_{x,t}(\alpha)$.
For the rest of the article we will assume $J\equiv0$ and will make massive use of 
the Hamiltonian function
$H\colon\R^n\times\R^n\times[0,T]\to \R$
\begin{equation}
H(\xi,p,s):=\min_{a\in A} \{p\cdot f(\xi ,a,s)+\ell(a,\xi,s)\}.
\end{equation}
Here we review the classic Hamilton-Jacobi-Bellman equation and 
give its proof in the general case in which there is an 
explicit temporal dependence in the Lagrangian.

\begin{theorem}[HJB] \label{th:HJB}
Let us assume that $D$ denotes the gradient operator with respect to x.
Furthermore, let us assume that $v\in
C^1(\R^n\times[0,T],\R)$ and that the minimum of $C_{x,t}$ exists for
every $x\in\R^n$ and for every $t\in[0,T]$. Then
$v$ solves the PDE
\begin{equation}
v_t(x,t)+H(x,Dv(x,t),t)=0,\qquad
(x,t)\in \R^n\times[0,T), 
\label{HJB}
\end{equation}
with terminal condition $v(x,T)=0,\quad \forall x\in\R^n$.
\end{theorem}

\begin{proof} 
Let $t\in[0,T)$ and $x\in\R^n$ be.
Furthermore, instead of the optimal
control let us use a constant control $\alpha_1(s)=a \in A$ for times $s\in[t,t+\epsilon]$
and then the optimal control for the remaining temporal interval.
More precisely let us denote with $y_{x,t}(\cdot, a)$ a solution of 
\begin{equation}
\begin{cases}
y'(s) = f(y(s),a,s) & s\in(t,T)\\
y(t)=x
\end{cases}
\end{equation}
and choose $\alpha_2\in\argmin_{\alpha\in\A} C_{y_{x,t}(t+\eps,
a),t+\eps}(\alpha)$. Now consider the following control
\begin{equation}
\alpha_3(s)=
\begin{cases}
\alpha_1(s) &  \text{if} \ s\in[t,t+\eps)\\
\alpha_2(s) &  \text{if} \ s\in[t+\eps,T] \ .
\end{cases}
\end{equation}
Then the cost associated to this control is
\begin{align}
C_{x,t}(\alpha_3)&=\int_t^{t+\eps} \ell(a, y_{x,t}(s,a),s)\,ds
+\int_{t+\eps}^T \ell(\alpha_2(s), y_{x,t}(s,\alpha_2),s)\,ds \\
&= \int_t^{t+\eps} \ell(a, y_{x,t}(s,a),s)\,ds
+v(y_{x,t}(t+\eps, a),t+\eps)
\end{align}
By definition of value function we also have that $v(x,t)\le
C_{x,t}(\alpha_3)$. When  rearranging this inequality, dividing by
$\eps$, and making use of the above relation we have
\begin{equation}
\frac{v(y_{x,t}(t+\eps, a),t+\eps)-v(x,t)}{\eps}+
\frac{1}{\eps}\int_t^{t+\eps} \ell(a, y_{x,t}(s,a),s)\,ds\ge0
\end{equation}
Now taking the limit as $\eps\to 0$ and making use of the fact that
$y'_{x,t}(t, a)=f(x,a,t)$ we get
\begin{equation}
v_t(x,t)+ Dv(x,t)\cdot f(x,a,t)+\ell(a,x,t)\ge0.
\end{equation}
Since this inequality holds for any chosen $a\in A$ we can say that
\begin{equation}
\inf_{a\in A} \{v_t(x,t)+ Dv(x,t)\cdot f(x,a,t)+\ell(a,x,t)\}\ge 0
\end{equation}
Now we show that the $\inf$ is actually a $\min$ and, moreover, that minimum is
$0$. To do this we simply choose
$\alpha^*\in\argmin_{\alpha\in\A} C_{x,t}(\alpha)$ and denote
$a^*:=\alpha^*(t)$, then 
\begin{equation}
v(x,t)= \int_t^{t+\eps} \ell(\alpha^*(s), y_{x,t}(s,\alpha^*),s)\,ds
+v(y_{x,t}(t+\eps, \alpha^*),t+\eps).
\end{equation}
Then again dividing by $\eps$ and using that
$y'_{x,t}(t, \alpha(t))=f(x,a^*,t)$ we finally get
\begin{equation}
v_t(x,t)+ Dv(x,t)\cdot f(x,a^*,t)+\ell(a^*,x,t)=0
\end{equation}
But since $a^*\in A$ and we knew that
$\inf_{a\in A} \{v_t(x,t)+ Dv(x,t)\cdot f(x,a,t)+\ell(a,x,t)\}\ge 0$
it means that
\begin{align}
\begin{split}
&\inf_{a\in A} \{v_t(x,t)+ Dv(x,t)\cdot f(x,a,t)+\ell(a,x,t)\}
= \\
&\min_{a\in A} \{v_t(x,t)+ Dv(x,t)\cdot f(x,a,t)+\ell(a,x,t)\}=0.
\end{split}
\end{align}
Recalling the definition of $H$ we immediately see that
the last inequality is exactly (HJB). 
\end{proof}

\section{The LQ problem}\label{Sec:LQ}
Now we consider the case of linear-quadratic control. In particular 
we assume that 
$f(x,a,s)=Ax+Ba$ and $\ell(a,x,s)=x\cdot Qx/2+ a\cdot Ra/2$ and
$J(x)=x\cdot Dx/2$, where $A$, $Q$, $D$ are in $\R^{n\times n}$,
$B\in\R^{n\times m}$ and $R\in\R^{m\times m}$, where $Q$, $R$, $D$
are non-negative definite and symmetric and $R$ is invertible.

We begin by noticing that 
\begin{align*}
\begin{split}
-R^{-1}B'Dv(x,t)=:\alpha(x,t)&\in\argmin_{a\in\R^m}
(Ax+Ba)\cdot Dv(x,t)+\frac{1}{2}x\cdot Qx+\frac{1}{2}a\cdot Ra \\
&=\Big\{-R^{-1}B'Dv(x,t)\Big\}.
\end{split}
\end{align*}
Hence HJB equation becomes
\begin{equation}
v_t(x,t)+Ax\cdot Dv(x,t)-\frac{1}{2}BR^{-1}B'Dv(x,t)\cdot Dv(x,t)
+\frac{1}{2}x\cdot Qx=0.
\end{equation}
At this point we guess that $v(x,t)=x\cdot \ricmat(t)x/2$, where
$\ricmat(T)=0$
so that with this choice we get (assuming to have chosen the initial
temporal instant to be 0)
\begin{equation}
x\cdot (\ricmat'(t)+2\ricmat(t)A-\ricmat(t)BR^{-1}B'\ricmat(t)+Q)x=0,\quad \forall x\in\R^n,
t\in(0,T),
\end{equation}
which implies
\begin{equation}
\begin{cases}
\ricmat'(t)-\ricmat(t)BR^{-1}B'\ricmat(t)+\ricmat(t)A+A'\ricmat(t)+Q=0&  t\in(0,T) \\
\ricmat(T)=D
\end{cases}
\label{eq2}
\end{equation}

\subsubsection*{SIFT: \textrm{Sign Flip Transformation}}
Now we introduce the basic idea for attacking LQ and, later on, LQT. 
Since Riccati equation is an autonomous system of equations, 
instead of using the final condition what we can change variables using
$$
\Phi\colon s\in [0,T]\to \tau\in[0,T], \qquad \tau:=T-s
$$
If we define $\hat \ricmat:= \ricmat\circ\Phi^{-1}$, we have
$\forall \tau\in[0,T]$ that $\hat \ricmat(\tau)= \ricmat(\Phi^{-1}(\tau))=
\ricmat(T-\tau)$.
In the new variable the Riccati equation becomes
\begin{align*}
\hat \ricmat'(\tau)&=-\ricmat'(T-\tau) \\
&=-\ricmat(T-\tau)BR^{-1}B'\ricmat(T-\tau)+\ricmat(T-\tau)A+A'\ricmat(T-\tau)+Q \\
&=-\hat \ricmat(\tau)BR^{-1}B'\hat \ricmat(\tau)+\hat \ricmat(\tau)A+A'\hat \ricmat(\tau)+Q, 
\end{align*}
with initial condition $\hat \ricmat(0)=\ricmat(T)=D$. Hence solving
\eqref{eq2} is equivalent to 
\begin{equation}
\begin{cases}
\hat \ricmat'(\tau)=-\hat \ricmat(\tau)BR^{-1}B'\hat \ricmat(\tau)+\hat \ricmat(\tau)A+A'\hat
\ricmat(\tau)+Q& \text{for} \ \tau\in(0,T) \\
\hat \ricmat(0)=D,
\end{cases}
\label{FlipLQ-eq}
\end{equation}
which can be solved forward in time. Throughout this paper, this is 
referred to as {\em flipped-sign Riccati's equation}.

\paragraph{The Case $m=n=1$.}
Let us consider the scalar case with terminal cost
$0$ (hence $D=0$). We get
\begin{equation}
\begin{cases}
\hat \ricmat'(\tau)=-B^2(\hat \ricmat(\tau))^2/R+2A \hat \ricmat(\tau)+Q& \tau\in(0,T) \\
\hat \ricmat(0)=0
\end{cases}.
\end{equation}
Now let $u(\tau):= (B^2/R)\hat \ricmat(\tau)$ so that $u$ solves 
\begin{equation}
\begin{cases}
u'(\tau)+(u(\tau))^2-2Au(\tau)-QB^2/R=0 & \tau\in(0,T) \\
\hat u(0)=0
\end{cases}.
\end{equation}
In order to solve this equation let us introduce $y$ to be any
$C^\infty([0,T]; \R)$ function 
such that $u(\tau)=y'(\tau)/y(\tau)$. With this parametrization of the solution we immediately see that $u'+u^2=y''/y$ hence the differential equation
becomes
\begin{equation}
\frac{y''(\tau)}{y(\tau)}-2A \frac{y'(\tau)}{y(\tau)}-\frac{QB^2}{R}=0 \ ,
\end{equation}
that should be solved with initial conditions $y'(0)=0$ and $y(0)=Y\ne0$. Therefore we are left with the second order linear ODE
\begin{equation}
y''(\tau) -2A y'(\tau)-\frac{QB^2}{R} y(\tau)=0 \ .
\end{equation}
The general solution to this equation is 
$y(\tau)=\gamma e^{\lambda_1 t}+ \delta e^{\lambda_2 t}$ with
$$
\lambda_1=A+\sqrt{A^2+\frac{QB^2}{R}},\qquad \lambda_2=A-\sqrt{A^2+
\frac{QB^2}{R}}.
$$
Using the initial conditions hence we get
$$
y(\tau)=Y\Bigl(\frac{\lambda_2}{\lambda_2-\lambda_1}e^{\lambda_1\tau}
-\frac{\lambda_1}{\lambda_2-\lambda_1}e^{\lambda_2\tau}
\Bigr),\quad
y'(\tau)=Y\frac{\lambda_1\lambda_2}{\lambda_2-\lambda_1}
\Bigl(e^{\lambda_1\tau} -e^{\lambda_2\tau}
\Bigr).
$$
Hence we have
\begin{equation}
\hat \ricmat(\tau)=\frac{R}{B^2}\lambda_1\lambda_2\frac{
e^{\lambda_1\tau} -e^{\lambda_2\tau}}{
\lambda_2e^{\lambda_1\tau} -\lambda_1 e^{\lambda_2\tau}.
}
\end{equation}
As $T\to\infty$ and $\tau\to\infty$ with $\tau<T$ we have that $\hat
\ricmat(\tau)\to \lambda_1 R/B^2$.

\begin{example}
Let us consider the following scalar example where 
$A=1,Q=3,S=1$. In this case the roots of algebraic Riccati's equation are
$\ricmat_1=-1$ and $\ricmat_2=3$. The forward solution of Riccati's ODE 
is unstable, whereas the forward solution of flipped-sign Riccati's ODE is
asymptotically stable and returns $\ricmat_1=-1$.
\end{example}

\section{The LQT problem}\label{sec:lqt}
The tracking problem assumes,  that given a signal $t\in[0,+\infty)\mapsto z(t)
\in\R^p$, the state of a linear system approaches it as much as possible. 
Formally, let 
$f(\xi,a,s)=A\xi+B a+Gz(s)$ and $\ell(a,x,s)=(x-Fz(s))\cdot Q(x-Fz(s))/2+
a\cdot Ra/2$ be and
$J(x)=0$, where $A$, $Q$  are in $\R^{n\times n}$,
$B\in\R^{n\times m}$ and $R\in\R^{m\times m}$, $G\in\R^{n\times p}$,
$F\in\R^{n\times p}$ where $Q$, $R$, 
are non-negative definite and symmetric and $R$ is invertible.
In what follows we will also assume $G=0$. \newline
In this case the Hamiltonian is
\begin{equation}
H(\xi,p,s)= A\xi\cdot p-\frac{1}{2}BR^{-1}B'p\cdot p
+\frac{1}{2}(\xi-Fz(s))\cdot Q(\xi-Fz(s)) \ ,
\end{equation}
and the Hamilton-Jacobi-Bellman equation becomes
\begin{equation}
v_t(x,t)+Ax\cdot Dv(x,t)-\frac{1}{2}BR^{-1}B'Dv(x,t)\cdot Dv(x,t)
+\frac{1}{2}(x-Fz(t))\cdot Q(x-Fz(t))=0 \ .
\end{equation}
Let us now make the guess
\begin{equation}
v(x,t)=\frac{1}{2}x\cdot \theta(t)x+\eta(t)\cdot x+ f(t) \ ,
\end{equation}
there $\theta(t)\in\R^{n\times n}$ for all $t\in[0,T]$ is symmetric
and $\eta(t)\in \R^n$ for all  $t\in[0,T]$, while $f$ is a scalar function of time. With this guess HJB equations becomes
\begin{equation*}
\frac{1}{2}x\cdot (\theta'(t)-\theta(t)S \theta(t)+\theta(t)A+A'\theta(t)+Q)x
+(\eta'(t)-(A'-\theta(t)S)\eta(t))-QFz(t))\cdot x=0 \ ,
\end{equation*}
for $(x,t)\in\R^n\times[0,T]$ and with the choice $f(t)=(1/2)\int_T^t
(\eta(s)\cdot S\eta(s)-Fz(s)\cdot QFz(s))\, ds$.
Since this relation must hold for all $x\in\R^n$ it is easy to show that
(see the scalar example below) it must be 
\begin{equation}
\begin{cases}
\theta'(t)-\theta(t)S \theta(t)+\theta(t)A+A'\theta(t)+Q=0 & \text{for} \ t\in[0,T) \\
\eta'(t)-(A'-\theta(t)S)\eta(t))-QFz(t)=0 & \text{for} \ t\in[0,T)
\end{cases} \ .
\end{equation}
Similarly we derive the boundary equations $\theta(T)=\eta(T)=0$.\\

\subsubsection*{Scalar LQ with Tracking}
Consider the scalar LQ tracking problem. Let $\ell(a,\xi,s)=Q(\xi-z(s))^2/2 + R a^2/2$ and choose $f(\xi,a,s)= A\xi+B a$. With this choices we have
\begin{equation}
H(\xi,p,s)=\frac{Q}{2}(\xi-z(s))^2-\frac{S}{2}p^2+ A\xi p,
\end{equation}
with $S\equiv B^2/R$. Let us make the guess
\begin{equation}
v(x,t)=\theta(t)x^2/2 +\eta(t) x +\frac{1}{2}\int_T^t S\eta^2(s)-Qz^2(s)\, ds
\end{equation}
for some unknown functions $\theta(t)$ and $\eta(t)$. Then since
$v_t(x,t)=\theta'(t) x^2/2+\eta'(t) x+(S\eta^2(t)-Qz^2(t))/2
$ and $Dv(x,t)=\theta(t) x(t)+\eta(t)$, the (HJB) equation becomes
\begin{equation}
\frac{1}{2}(\theta'(t)-S\theta^2(t)+2A\theta(t)+Q)x^2+(\eta'(t)
+(A-S\theta(t))\eta(t)-Qz(t))x=0,
\end{equation}
which holds for all $(x,t)\in \R^n\times[0,T)$. Now we can differentiate this
expression with respect to $x$ and get $\forall (x,t)\in \R^n\times[0,T)$
\begin{equation}
(\theta'(t)-S\theta^2(t)+2A\theta(t)+Q)x+\eta'(t)
+(A-S\theta(t))\eta(t)-Qz(t)=0.
\end{equation}
Since this relation holds for all $x\in\R^n$, in particular it must hold for $x=0$, which tells us that 
\begin{equation}
\eta'(t) +(A-S\theta(t))\eta(t)-Qz(t)=0\quad \forall t\in[0,T).
\end{equation}
For all other $x\ne0$ and $t\in[0,T)$ we get
\begin{equation}
(\theta'(t)-S\theta^2(t)+2A\theta(t)+Q)x=0,
\end{equation}
which implies
\begin{equation}
\theta'(t)-S\theta^2(t)+2A\theta(t)+Q=0.
\end{equation}
Moreover the boundary condition $v(x,T)=0$ for all $x\in\R^n$ similarly implies that $\theta(T)=\eta(T)=0$. This defines the following system of differential equations
\begin{equation}
\begin{cases}
\theta'(t)-S\theta^2(t)+2A\theta(t)+Q=0;& \text{for} \ t\in[0,T) \\
\eta'(t) +(A-S\theta(t))\eta(t)-Qz(t)=0;& \text{for} \ t\in[0,T) \\
\theta(t)=\eta(t)=0.& \text{for} \ t=T
\end{cases}
\label{eq3}
\end{equation}

\subsubsection*{SIFT: Sign Flip Transformation}
When using the sign flip transformation, we get 
\begin{equation}
\begin{cases}
\hat \theta'(\tau)+S\hat\theta^2(\tau)-2A\hat\theta(\tau)-Q=0;& \text{for} \ \tau\in(0,T] \\
\hat \eta'(\tau) -(A-S \hat \theta(\tau))\hat \eta(\tau)+Q\hat z(\tau)=0;& \text{for} \ \tau\in(0,T] \\
\hat\theta(\tau)=\hat\eta(\tau)=0,& \text{for} \  \tau=0
\end{cases}
\label{eq4}
\end{equation}
where $\hat z(\tau)=z(T-\tau)$.
Notice that these equations are equivalent to Eq.~(\ref{FlipLQ-eq}) for 
LQ. In particular, we get the same Riccati's equation for $\hat\theta$.
However, in order to determine $\hat \eta$ we need the knowledge of
$z(T-\tau)$, that is we need to know the value of the tracking signal
beginning from the end of the horizon. 

\subsubsection*{Kinetic terms}
In many problems of interest, including 
learning, we would like to exert a control that
has as few temporal variations as possible. This can be 
achieved by ``controlling the variations of the control''.
Consider for instance the following system
\begin{equation}\begin{cases}
\begin{pmatrix}y_1'(s)\\
y_2'(s)\end{pmatrix}=
\begin{pmatrix}a_{11}& a_{12}\\0&0\end{pmatrix}
\begin{pmatrix}y_1(s)\\
y_2(s)\end{pmatrix}+
\begin{pmatrix}0\\
b\end{pmatrix}\alpha(s) & \hbox{for $s\in(t,T]$}\\
\noalign{\medskip}
\begin{pmatrix}y_1(s)\\
y_2(s)\end{pmatrix}=
\begin{pmatrix}x_1\\
x_2\end{pmatrix}=:x & \hbox{for $s=t$}\\
\end{cases}
\label{eq:kinetic-sys}
\end{equation}
Notice that in the above system the control is the temporal derivative of
$y_2$ which, in the first equation plays the role of the control term in the
scalar version of the LQT problem. In this way we can directly
penalize the variations of $y_2$ with the following cost functional.
\begin{equation}\medmuskip -1mu
\begin{aligned}
C_{x,t}(\alpha)=\frac{1}{2}
\int_t^T &
\left(
\begin{pmatrix}{y_1}_{x,t}(s,\alpha)\\
{y_2}_{x,t}(s,\alpha)\end{pmatrix}-\begin{pmatrix}1\\
0\end{pmatrix} z(s)
\right)\cdot
\begin{pmatrix}
q&0\\0&0\\
\end{pmatrix}
\left(
\begin{pmatrix}{y_1}_{x,t}(s,\alpha)\\
{y_2}_{x,t}(s,\alpha)\end{pmatrix}-\begin{pmatrix}1\\
0\end{pmatrix} z(s)
\right)\\
\noalign{\medskip}
&\qquad\qquad\ +
R(\alpha(s))^2\, ds\end{aligned}
\label{eq:kinetic-cost}
\end{equation}
Equations~\eqref{eq:kinetic-sys} and~\eqref{eq:kinetic-cost} are still in the
form of an LQT problem (with scalar reference signal)
described in Section~\ref{sec:lqt} with the choices $p=1$, $m=1$, $n=2$,
$A=({a_{11}\atop 0} {a_{12}\atop 0})$, $B={0\choose b}$,
$Q=({q\atop 0} {0\atop 0})$, $F={1\choose 0}$, $R\in\R\setminus\{0\}$.

\section{SIFT and forward approximation for LQT}
In this Section we introduce a forward method to compute an approximate solution of the LQT problem. We focus on the scalar case discussed in Section~4. \newline
The main feature of this method is that future information about the reference signal $z(s)$ or the costate $p(s)$ is not involved for computing the trajectories of the state $y(s)$ and the costate $p(s)$, for $s \in [0, T]$. Therefore, this solution aims to approximate the optimal one updating the trajectories only forward in time. \newline
\newline
The core of the forward approximate solution can be described as follows. Recalling that
\begin{equation}
p(t) = Dv(x,t) = \theta(t) y(t) + \eta(t) \ \ \text{for} \ t \in [0, T] \ , 
\end{equation}
where $\theta(t)$ and $\eta(t)$ satisfy the system of differential equations $\eqref{eq3}$, one can provide approximations $\tilde \theta(\tau)$ and $\tilde \eta(\tau)$ for $\theta(\tau)$ and $\eta(\tau)$ respectively, solving the following system of differential equations
\begin{equation}
\begin{cases}
\tilde \theta'(\tau)+S\tilde \theta^2(\tau)-2A\tilde \theta(\tau)-Q=0& \text{for} \ \tau \in(0,T] \\
\tilde \eta'(\tau) -(A-S \tilde \theta(\tau))\tilde \eta(\tau)+Q z(\tau)=0& \text{for} \ \tau \in(0,T] \\
\tilde \theta(\tau)=\tilde \eta(\tau)=0& \text{for} \  \tau=0
\end{cases} \ ,
\label{eq5}
\end{equation}
forward in time. Notice that the equation for $\tilde \theta(\tau)$ is equivalent to the equation for $\hat \theta(\tau)$ in $\eqref{eq4}$, which is the Riccati equation and it is independent of the reference signal. Similarly to what we have done in Section~3, it can be proved that $\hat \theta(\tau)$ converges to a constant value for $\tau \to \infty$. Therefore, $\tilde \theta(\tau) = \hat \theta(\tau)$ approximates $\theta(\tau)$ for long time horizons $T$.\newline
On the other hand, the equation for $\tilde \eta(\tau)$ differs from the one for $\hat \eta(\tau)$ by the fact that the former involves the reference signal $z$ at the present time $\tau$, while the second one involves $z(T-\tau)$. This implies that $\tilde \eta(\tau)$ reproduces the dynamics of $\eta(\tau)$ (and not $\hat \eta(\tau)$), starting from the initial condition $\tilde \eta(0)=0$. \newline
Since the initial condition for $\tilde \eta$ is (in principle) different from the optimal one for $\eta$, its trajectory represents an approximation of the optimal $\eta$. \newline
\newline
The forward approximate solution is computed as follows. For each time instant $\tau \in (0,T]$, we compute $\tilde \theta (\tau)$ and $\tilde \eta(\tau)$ forward in time, obtaining an approximation $\tilde p (\tau)$ of $p(\tau)$ by
\begin{equation}
\tilde p (\tau) = \tilde \theta (\tau) \tilde y(\tau) + \tilde \eta(\tau) \ .
\end{equation}
Recalling that the optimal control is  $\alpha^{*}(\tau) = -(B/R) p(\tau)$
inserting $\tilde \alpha(\tau) = -(B/R) \tilde p(\tau) $ in
\begin{equation}
\begin{cases}
y'(\tau)=A y(\tau) + B \alpha(\tau), \ \text{for} \ \tau \in \left[0, T \right]; \\
y(0)=x,
\end{cases}
\end{equation}
we get an approximation $\tilde y$ for the optimal trajectory of the state $y$. \newline
In Section~6, we will report the experimental results of this forward approach for different reference signals. 

\section{Experiments} \label{Sec:exp}
In this Section we report the experimental part of our work, which aims to
study the forward approximate solution of the scalar LQT for different
reference signals, with respect to different time horizons $T$ and values of $R$. The differential equations for $y$, $\tilde\theta$, and
$\tilde \eta$ are discretized using the Euler forward approximation.

Setting a time horizon $T$ and a frequency $f$, the reference signals used in
the experiments are defined as follows:
\begin{align}
\nonumber
z_1(s) &= \begin{cases} -5 \cos( 2 \pi f s) & \text{if } s \in [0, T/2]; \\
z_1(T/2) & \text{if } s \in (T/2, T],\end{cases} \\
\nonumber
z_2(s) &= \begin{cases} - 10T^{-1}(- s + T/2)[\cos(2 \pi f s ) + 0.3 \cos(5
\pi f s)] & \text{if } s \in [0, T/2);\\ 1/(1+ \exp(-s + 2T/3))& \text{if } s
\in [T/2, T],\end{cases}\\
\nonumber
z_3(s) &= \begin{cases}10T^{-1}( - s +T/2)(\cos( 2 \pi f s))/(2 \pi f s + 1) &
\text{if } s \in [0, T/2);\\
1 & \text{if } s \in [T/2, T].\end{cases}
\end{align}
It is important to notice that $z_1(s)$, $z_2(s)$ and $z_3(s)$ are
different with respect to properties of continuity and periodicity in
$[0, T]$. More precisely:
\begin{itemize}
    \item $z_1$ is continuous for every $s \in [0, T]$ and it is periodic with period $1/f$ for $s \in[0, T/2]$. For $s \in (T/2, T]$, it is constant; 
    \item $z_{2}$ is non-periodic with a discontinuity in $s=T/2$. For $s \in [T/2, T]$, it is a sigmoidal function; 
    \item $z_{3}$ is non-periodic with a discontinuity in $s=T/2$. For $s \in [T/2, T]$, it is constant. 
\end{itemize} 

These choices for $z_1$, $z_2$ and $z_3$ aim to cover a variety of
reference signals for generalization purposes.
Let us set the frequency $f$ equal to $f=0.02$ Hz. 
Considering the scalar LQT introduced in Section~$\ref{sec:lqt}$, with $A=B=Q=1$ and initial condition $y(0)= x = 2$, it is possible to compare the forward approximate solution with the optimal one. For this purpose, we report in Table~\ref{tab1}
the values of the average cost functional for both the optimal and the forward solutions, for different values of $R$ and time horizons $T$. The average is computed dividing the cost functional by the corresponding time horizon $T$.

We can compute the
percentage error $\PE(\tilde \alpha, \alpha^{*}) $
between two correspondent values of average cost functional as
follows:
\begin{equation}\label{eq:PE}
\PE(\tilde \alpha,\alpha^{*}) := \frac{C_{x,0}(\tilde \alpha)
- C_{x,0}(\alpha^{*})}{C_{x,0}(\alpha^{*})}\times 100,
\end{equation}
where $\alpha^{*}$ is the optimal control and  $\tilde \alpha$ is the
control obtained by the forward approximation method.
Since $C_{x,0}(\tilde \alpha)\ge
C_{x,0}(\alpha^{*})$, the percentage error is always non-negative.
In Table \ref{tab1}, we reported the percentage error between the values of the average cost functional for the optimal solution and the forward approximation.

\begin{table}[t]
\centering\small
\caption{Values of the average cost functional obtained
with the optimal and forward  solutions varying $R$ and $T$, for the
different reference signals $z_1$, $z_2$, $z_3$ with fixed frequency $f=.02$
Hz and its percentage error as defined in Eq.~\eqref{eq:PE}}. \label{tab1}
\setlength{\tabcolsep}{2pt}
\begin{tabular}{c@{\hskip 10pt}ccccccccccccc}
\toprule
&&&&Optimal&&&&Forward&&&&PE\%&\\ \cmidrule{4-6}\cmidrule{8-10}\cmidrule{12-14}
$T$ & $R$ && $z_1$ & $z_2$ & $z_3$              &&$z_{1}$ & $z_{2}$ & $z_{3}$ &&$z_{1}$ & $z_{2}$ & $z_{3}$   \\
\midrule                                       
       & 8e-4 && 0.0306 & 0.0433 & 0.0065     &&0.0404 & 0.0576 & 0.0087    && 32.02   & 33.03  & 33.85  \\ 
25 s & 0.01 && 0.1296 & 0.1684 & 0.0316       &&0.1658 &  0.2207 & 0.0398   && 27.93  & 31.06  & 25.95   \\
       & 1 && 2.2223 & 2.0674 & 0.7386          &&2.8346 & 2.8804 & 1.5665    && 27.55  & 39.32  & 112.09      \\
\midrule                                                                                                  
         & 8e-4 && 0.0104 & 0.0054 & 0.0009   &&0.0116 & 0.0069 & 0.0012    && 11.53  & 27.78  & 33.33    \\ 
250 s  & 0.01 && 0.1027 & 0.0278 & 0.0063     &&0.1082 & 0.0339 & 0.0071    && 5.36  & 21.94  & 12.69      \\
         & 1 && 4.7220 & 0.7464 & 0.2280        &&4.8580 & 0.8462 & 0.3062    && 2.88  & 13.37  & 34.29       \\
\midrule                                                                                                  
         & 8e-4 && 0.0079 & 0.0016 & 0.00029  &&0.0082 & 0.0019 & 0.00032   && 3.80  & 18.75  & 10.34    \\
2e3 s & 0.01 && 0.0945 & 0.0150 & 0.0030     &&0.0969 & 0.0166 & 0.0031    && 2.54  & 10.67  & 3.33      \\
         & 1 && 4.7030 & 0.6677 & 0.1402        &&4.7620 & 0.7016 & 0.1500    && 1.26  & 5.08  & 6.99         \\
\bottomrule
\end{tabular}
\end{table}

From Table $\ref{tab1}$, it is possible to see that the forward method
produces a larger percentage error on shorter time horizons. The reason for
this behavior is what we mentioned in Section~5. Since the initial condition
for $\tilde \eta$ is different from the optimal one of $\eta^{*}$, the
forward solution approximates worse at the beginning of the time
interval. Since the average cost functional is computed dividing the integral
by $T$, the effect of the initial error is stronger on short time horizons.

On the other hand, after an initial gap, the forward approximation reproduces
the optimal solution in a reliable way. Indeed, when $T$ grows, the
percentage error reduces. \newline We can also notice that, for longer time
horizons, the impact of $R$ on the percentage error reduces. This result can
be explained by the same reason. Indeed, increasing $R$ the costate is
subject to larger variations and the same holds for $\eta$. Therefore, the
initial gap is emphasized and the percentage error is higher on shorter
horizons. In Figure \ref{Fig1} we report the trajectories of the state $y$
and the costate $p$, with $T = 250$ s for the different signals, varying the
value of $R$. The reference signal is the green solid line, the state $y$ is the black dashed
line, the costate $p$ is the orange dashed line.

\def\<#1>{\vcenter{\hbox{\includegraphics{./figs/fig1-double-#1.mps}}}}
\def\boxit#1{\vbox{\hrule\hbox{\vrule\kern3pt
\vbox{\kern3pt#1\kern3pt}\kern3pt\vrule}\hrule}}

\begin{figure}[t]
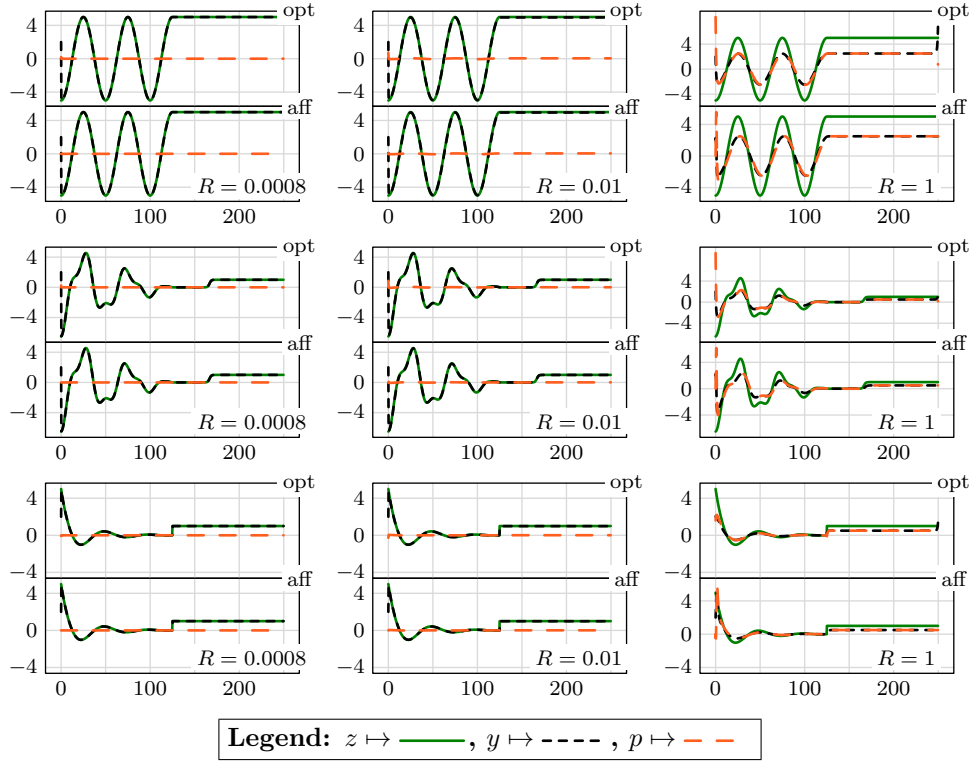

\centerline{\includegraphics{./figs/fig1-double-0.mps}
\hfil\includegraphics{./figs/fig1-double-1.mps}
\hfil\includegraphics{./figs/fig1-double-2.mps}}

\centerline{\includegraphics{./figs/fig1-double-3.mps}
\hfil\includegraphics{./figs/fig1-double-4.mps}
\hfil\includegraphics{./figs/fig1-double-5.mps}}
\centerline{\includegraphics{./figs/fig1-double-6.mps}
\hfil\includegraphics{./figs/fig1-double-7.mps}
\hfil\includegraphics{./figs/fig1-double-8.mps}}
\medskip
\hbox to \hsize{\hfil\boxit{\hbox{\bf Legend: $z\mapsto\<990>$,
$y\mapsto\<991>$, $p\mapsto\<992>$}}\hfil}

\caption{Comparison between the optimal solution (``opt'') and the forward
approximate solution (``'aff') for the different signals $z_{1}$ (first row),
$z_{2}$ (second row) and $z_{3}$ (third row)
and the different values of $R$. The time horizon $T$ is set equal to 250
s. The first row of each image is the optimal solution, while the second row
is the forward approximation.}  \label{Fig1}
\end{figure}

\subsubsection*{Comparison with MPC}
In this subsection we compare the proposed forward solution with Model Predictive Control (MPC) approach for solving the LQT problem considered up to now. \newline
Firstly, it is important to remark that MPC requires future information on the reference signal, whereas the proposed forward approximation involves the value of $z$ at present time. Indeed, let us discretize the time domain $\left[0, T\right]$ in $n$ points and call $\delta s = T / n$. Being at time $s$, MPC works as follows:
\begin{enumerate}
\item at time $s$ and for the current state $y(s)$, solve an optimal control problem over a fixed future interval $\left[s, \ s + \left(w-1 \right)\delta s \right]$, considering the current and future constraints. Here, $w$ is the fixed depth of the future interval that is considered; 
\item apply only the first step in the resulting optimal control sequence;
\item measure the reached state $y(s + \delta s)$;
\item repeat the fixed horizon optimization at time $s + \delta s$ over the interval $\left[s + \delta s, \ s + w\delta s \right]$. 
\end{enumerate}
For solving the optimal control problem at point 1, future information about
the reference signal in the time interval $\left[s, \ s + \left(w-1
\right)\delta s \right]$ is needed. In order to compare the two approaches,
it is reasonable to provide the same amount of information to both the MPC
and the proposed forward approaches. Therefore, let us set the size of the
temporal window to $w=2$, so that the future interval considered in MPC is
equal to $\left[s, s+\delta s \right]$. In this way, we make use of MPC for
computing the solution in an online manner, as we do in our approach. In
Table $\ref{tab3}$ we report the values of the average cost functional
computed for the MPC solution, for different values of $T$ and $R$. We
denoted as ``$*$'' the values of functional greater than $10^{4}$.

\begin{wraptable}{r}{10pc}
\setlength{\tabcolsep}{2pt}
\centering\small
\caption{Values of the average cost functional obtained with the MPC solution
varying $R$ and $T$, for the different reference signals $z_{1}$, $z_{2}$,
$z_{3}$ with fixed frequency $f=.02$ Hz and $w=2$.} \label{tab3}
\begin{tabular}{ccccc}
\toprule
$T$ & $R$ & $z_1$ & $z_2$ & $z_3$\\
\midrule                                       
       & 8e-4 & $*$ & $*$ & $*$ \\ 
{25 s} & 0.01 & $*$ &  $*$ & $*$ \\
       & 1 & $*$ & $*$ & $*$ \\
\midrule                                                                  
         & 8e-4 & 2.084 & 0.319 & 0.103 \\ 
{250 s}  & 0.01 & $*$ & $*$ & $*$ \\
         & 1 & $*$ & $*$ & $*$ \\
\midrule                                                                  
         & 8e-4 & 2.085 & 0.296 & 0.062 \\
{2e3 s} & 0.01 & $*$ & $*$ & $*$ \\
         & 1 & $*$ & $*$ & $*$ \\
\bottomrule
\end{tabular}
\end{wraptable}

Comparing Tables $\ref{tab1}$ and $\ref{tab3}$ it is possible to state that, without future information on the reference signals, the forward solution approximates the optimal one way better than the MPC solution, which diverges for larger values of $T$ and $R$.

As one could expect, the approximation with the MPC solution becomes more reliable when future information is added (i.e., the depth of the future interval $w$ is increased). On the other hand, the computational cost for applying MPC grows. Indeed, if one increases $w$ for having similar results to our approach and measures the time needed by the machine for computing the approximate solution, it is possible to see that MPC requires much more time. In Table $\ref{tab4}$ we report the time for computing the approximate solution with the proposed forward approach and MPC, with reference signal $z_{1}$, $T=250$ s and $R=8e-4, \ 0.01, \ 1$. For having comparable values of the percentage error, we set $w=17, \ 85, \ 975$, for $R=8e-4, \ 0.01, \ 1$ respectively.

\begin{table}[t]
\centering\small
\caption{Cost functional, percentage error and computational time for
computing the approximate solution by MPC and by the proposed forward
approach. The reference signal is $z_{1}$ with fixed frequency $f=.02$ Hz,
the time horizon is $T=250$ s.} \label{tab4}
\begin{tabular}{cccccccccc}
\toprule
&&&&MPC&&&&Forward&\\ \cmidrule{4-6}\cmidrule{8-10}
$T$ & $R$ & $w$ & Cost & PE\% & Time              && Cost & PE\% & Time \\
\midrule                                       
       & 8e-4 & 17 & 0.0115 & 10.57 & 6.03 s     &&0.0116 & 11.53 & 0.47 s  \\ 
{250 s} & 0.01 & 85 & 0.1080 & 5.16 & 29.07 s     &&0.1082 & 5.36 & 0.42 s  \\
       & 1 & 975 & 4.8560 & 2.84 & 329.89 s       &&4.8580 & 2.88 & 0.41 s  \\
\bottomrule
\end{tabular}
\end{table}

For longer time horizons $T$, the depth $w$ of the temporal window must be increased further and the computational time becomes higher.

We can conclude that approximating the optimal solution with the proposed forward approach gives more reliable results than using MPC, when future information about the reference signals is not provided and proceeding in an online manner. Moreover, the proposed forward approach is more efficient from a computational point of view.
\section{Conclusions}
\label{Sec:conc}
In this paper we have proposed a forward on-line approach for 
the LQT problem which provides accurate approximations of the 
optimal solution especially when the cost functional is mostly 
weighing the precision term. We have given experimental evidence
of a dramatic computational improvement with respect MPC 
on comparable approximations of the optimal strategy. 
Further investigations concerning the multidimensional case 
and the formalization of this approach in the discrete setting
of computation. \\
~\\
\noindent {\bf \Large Acknowledgement}\\
~\\
We thank Gianni Bianchini, Andrea Garulli, and Lapo Faggi for insightful discussions.

\bibliographystyle{plain} 
\bibliography{refs}

\begin{thebibliography}{1}

\bibitem{evans2010partial}
Lawrence~C Evans.
\newblock {\em Partial differential equations}, volume~19.
\newblock American Mathematical Soc., 2010.

\bibitem{REBLE20121812}
Marcus Reble and Frank Allgöwer.
\newblock Unconstrained model predictive control and suboptimality estimates
  for nonlinear continuous-time systems.
\newblock {\em Automatica}, 48(8):1812--1817, 2012.

\end{thebibliography}

\end{document}